\documentclass[a4paper]{amsart}

\usepackage[backrefs,lite]{amsrefs}
\usepackage{mathtools}   
\usepackage{amssymb}              
\usepackage{lmodern,bm}              
\usepackage{longtable}
\usepackage{tikz}
\usepackage[all]{xy}
\usepackage{dsfont}
\usepackage{pdflscape}
\usepackage{bigdelim}
\usepackage{multirow}
\usepackage{cancel}
\usepackage{array}
\usepackage[hidelinks]{hyperref}
\usepackage{stmaryrd}
\usetikzlibrary{cd}

\CompileMatrices

\newtheorem{theorem}{Theorem}[section]
\newtheorem{introthm}{Theorem}
\newtheorem{lemma}[theorem]{Lemma}
\newtheorem{proposition}[theorem]{Proposition}
\newtheorem{corollary}[theorem]{Corollary}

\theoremstyle{definition}

\newtheorem{definition}[theorem]{Definition}

\numberwithin{equation}{theorem}

\def\vector2#1#2{\left(\begin{array}{c} #1 \\ #2 \end{array}\right)}

\def\Cl{{\rm Cl}}

\def\CC{{\mathbb C}}

\def\ZZ{{\mathbb Z}}
\def\RR{{\mathbb R}}

\def\QQ{{\mathbb Q}}
\def\PP{{\mathbb P}}

\def\conv{{\rm conv}}

\def\bangle#1{{\langle #1 \rangle}}

\def\rk{{\rm rk}}

\def\Hom{{\rm Hom}}

\def\lcm{{\rm lcm}}

\def\vol{{\rm vol}}
\def\Vol{{\rm Vol}}
\def\red{{\rm red}}

\def\syl{{\rm syl}}

\makeatletter
\newcommand*{\defeq}{\mathrel{\rlap{%
                     \raisebox{0.3ex}{$\m@th\cdot$}}%
                     \raisebox{-0.3ex}{$\m@th\cdot$}}%
                     =}
\makeatother

\title[Sharp degree bounds for fake weighted projective spaces]{Sharp degree bounds for fake weighted projective spaces}

\author[Andreas Bäuerle]{Andreas Bäuerle}

\address{Mathematisches Institut, Universität Tübingen,
Auf der Morgenstelle 10, 72076 Tübingen, Germany}
\email{baeuerle@math.uni-tuebingen.de}


\begin{document}

\begin{abstract}
We give sharp upper bounds on the anticanonical degree of fake weighted projective spaces, only depending on the dimension and the Gorenstein index.
\end{abstract}

\maketitle

\section{Introduction}

A $d$-dimensional \emph{fake weighted projective space}
is a quotient $X = (\CC^{d+1} \backslash \{0\}) / G$ by a
diagonal action of $G := \CC^*\!\times\!\Gamma$, 
where $\Gamma$ is a finite abelian group and the
factor~$\CC^*$ acts via positive weights.
Any fake weighted projective space~$X$ is
normal,~$\QQ$-factorial, of Picard number one
and is a Fano variety, i.e. its anticanonical
divisor~$-\mathcal{K}$ is ample.
Apart from the classical projective spaces, all
fake weighted projective spaces are singular,
but have at most abelian quotient singularities.

Fake weighted projective spaces form an interesting
example class for the general question of effectively
bounding geometric data of a Fano variety in terms
of its singularities.
For instance, Kasprzyk~\cite{Ka} bounds the order
of the torsion part of the divisor class group 
of a fake weighted projective space $X$ provided
that it has at most canonical singularities.
Another invariant of the singularities is the
\emph{Gorenstein index}, i.e., the minimal
positive integer $\iota$ such that 
$\iota \mathcal{K}$ is Cartier.
In the case of Gorenstein index $\iota = 1$,
Nill~\cite{Ni} provides a bound for
the \emph{degree} of a $d$-dimensional fake weighted
projective space $X$, i.e., the self intersection
number $(-\mathcal{K})^d$ of its anticanonical
divisor.

In the present paper, we extend Nill's bound to higher
Gorenstein indices.
For any $d \ge 2$ define a $d+1$ tuple of positive integers
by 
$$
Q_{\iota,d}
 := 
\left(
\frac{2 t_{\iota,d}}{s_{\iota,1}},
\ldots ,
\frac{2 t_{\iota,d}}{s_{\iota,d-1}}, 1, 1
\right),
\ \
s_{\iota,k} := \iota\, s_{\iota,1} \cdots s_{\iota,k-1}+1,
\ \
t_{\iota,k} := \iota\, s_{\iota,1} \cdots s_{\iota,k\!-\!1},
$$
where $s_{\iota,1} := \iota + 1$. Our main result
provides sharp upper bounds on the degree $(-\mathcal{K})^d$ in terms
of the Gorenstein index and lists the cases attaining these bounds:

\begin{introthm}\label{thm:main}
The anticanonical degree of any $d$-dimensional fake weighted projective space $X$ of Gorenstein index $\iota$ is bounded according to the following table.
\begin{longtable}{c|c|cc|cc|c}
    $d$
    &
    $1$
    &
    $2$
    &
    $2$
    &
    $3$
    &
    $3$
    &
    $\ge 4$
    \\[2pt]\hline
    &&&&&\\[-9pt]
    $\iota$
    &
    $\ge 1$
    &
    $1$
    &
    $\ge\!2$
    &
    $1$
    &
    $\ge\!2$
    &
    $\ge 1$
    \\[2pt]\hline
    &&&&&\\[-10pt]
    $\begin{array}{c}
         \text{bound on}\\
         (-\mathcal{K})^d
    \end{array}$
    &
    $2$
    &
    $9$
    &
    $\frac{2(\iota+1)^2}{\iota}$
    &
    $72$
    &
    $\frac{2\,t_{\iota,3}^2}{\iota^4}$
    &
    $\frac{2\,t_{\iota,d}^2}{\iota^{d+1}}$
    \\[8pt]\hline
    &&&&&\\[-10pt]
    $\begin{array}{c}
         \text{attained}\\
         \text{exactly by}
    \end{array}$
    &
    $\PP^1$
    &
    $\PP^2$
    &
    $\PP(2\iota,1,1)$
    &
    $\begin{array}{l} \PP(3,1,1,1), \\ \PP(6,4,1,1) \end{array}$
    &
    $\PP(Q_{\iota,3})$
    &
    $\PP(Q_{\iota,d})$
\end{longtable}
\noindent Equality on the degree holds if and only if $X$ is isomorphic to one of the weighted projective spaces in the last row of the table.
\end{introthm}

The article is organized as follows. Section \ref{section:fpws-ws} provides basic properties of fake weighted projective spaces. In Section \ref{section:gi-uf} we assign to any $d$-dimensional fake weighted projective space of Gorenstein index $\iota$ a certain partition of $1/\iota$ into $d+1$ unit fractions and give a formula to compute the anticanonical degree in terms of the denominators of these unit fractions. Section \ref{section:syl-seq} contains the number theoretic part of the proof of Theorem \ref{thm:main}. In Section \ref{section:proof} we complete the proof of the main result. This amounts to constructing a weighted projective space of given dimension $d$ and Gorenstein index $\iota$ whose unit fraction partition of $1/\iota$ meets a maximality condition.

\section{Fake weighted projective spaces}\label{section:fpws-ws}

We recall basic properties of fake weighted projective spaces and fix our notation, see also \cite[Sect.~3]{Ni}. The reader is assumed to be familiar with the very basics of toric geometry \cite{CoLiSch, Fu}. Throughout the article $N$ is a rank $d$ lattice for some $d \in \ZZ_{\ge 2}$. Its dual lattice is denoted by $M = \Hom(N,\ZZ)$ with pairing $\bangle{\cdot\, , \cdot } \colon M \times N \rightarrow \ZZ$. We write $N_\RR := N \otimes_\ZZ \RR$ and $M_\RR := M \otimes_\ZZ \RR$. Polytopes $P \subseteq N_\RR$ are assumed to be full dimensional with $\mathbf{0} \in N_\RR$ in their interior. The \emph{normalized volume} of a $d$-dimensional polytope $P$ is $\Vol(P) = d! \vol(P)$, where $\vol(P)$ denotes its euclidean volume. The \emph{dual} of a polytope $P \subseteq N_\RR$ is the polytope
\begin{equation*}
    P^* \ := \ \{ u \in M_\RR ; \, \bangle{u,v} \ge -1 \text{ for all } v \in P\} \ \subseteq \ M_\RR.
\end{equation*}
For a facet $F$ of $P$ we denote by $u_F \in M_{\RR}$ the unique linear form with $\bangle{u_F,v} = -1$ for all $v \in F$. We have
\begin{equation*}
    P^* \ = \ \conv(\, u_F; \, F \text{ facet of } P\, ), \qquad P \ = \ \{v \in N_\RR ; \, \bangle{u_F,v} \ge -1,\, F \text{ facet of } P\}.
\end{equation*}
A \emph{lattice polytope} $P \subseteq N_\RR$ is a polytope whose vertices are lattice points in $N$. We regard two lattice polytopes $P \subseteq N_{\RR}$ and $P' \subseteq N'_{\RR}$ as isomorphic if there is a lattice isomorphism $\varphi \colon N \rightarrow N'$ mapping $P$ bijectively to $P'$.

\begin{proposition}
The fake weighted projective spaces are precisely the toric varieties $X = X(P)$ associated to the face fan of a lattice simplex $P \subseteq N_\RR$ with primitive vertices.
\end{proposition}

\begin{proof}
This follows from \cite[Sect.~5.1]{CoLiSch}.
\end{proof}

Two fake weighted projective spaces are isomorphic if and only if their corresponding lattice simplices are isomorphic. The (true) weighted projective spaces among them correspond to lattice simplices whose vertices generate the lattice. Many geometric properties of a fake weighted projective space can be read off the corresponding lattice simplex. Here we focus our attention on the Gorenstein index and the anticanonical degree.

\begin{definition}
The \emph{index} of a lattice polytope $P \subseteq N_\RR$ is the positive integer
\begin{equation*}
    \iota_P \ := \ \min ( \, k \in \ZZ_{\ge 1} ; \, k P^* \text{ is a lattice polytope}\, ).
\end{equation*}
\end{definition}

\begin{lemma}\label{lemma:gorind}
The Gorenstein index of any fake weighted projective space $X = X(P)$ equals the index $\iota_P$ of the corresponding lattice simplex $P \subseteq N_\RR$.
\end{lemma}

\begin{proof}
The dual polytope $P^*$ is the polytope associated to $-\mathcal{K}_X$, see \cite[Sect.~4.3]{CoLiSch}. The assertion follows from \cite[Thm.~4.2.8]{CoLiSch}.
\end{proof}

\begin{lemma}{\rm{(cf.~\cite[p.~111]{Fu})}}\label{lemma:degvol}
Let $X = X(P)$ a $d$-dimensional fake weighted projective space. Then we have $(-\mathcal{K}_X)^d \ = \ \Vol(P^*)$.
\end{lemma}

Any weighted projective space $\PP(q_0,\dots,q_d)$ is up to an isomorphism uniquely determined by its weights $(q_0,\dots,q_d)$. More generally we assign weights to any lattice simplex $P \subseteq N_\RR$.

\begin{definition}{(cf.~\cite{Co,Ni})}\label{def:weightsys}
A \emph{weight system $Q$ (of length $d$)} is a $(d+1)$-tuple $Q = (q_0, \dots, q_d)$ of positive integers. We call
\begin{equation*}
    |Q| \ := \ q_0 + \dots + q_d, \qquad \lambda_Q \ := \ \gcd(Q), \qquad Q_{\red} := Q/\lambda_Q
\end{equation*}
the \emph{total weight}, the \emph{factor} and the \emph{reduction} of $Q$. A weight system $Q$ is called \emph{reduced} if it coincides with its reduction and it is called \emph{well-formed} if we have $\gcd(q_j \, ; \, j = 0,\dots,d,\ j \ne i ) = 1$ for all $i = 0,\dots, d$.
\end{definition}

\begin{definition}{(cf.~\cite{Co,Ni})}
To any lattice simplex $P = \conv(v_0,\dots,v_d) \subseteq N_\RR$ we associate a weight system by
\begin{equation*}
    Q_P \ := \ (q_0,\dots,q_d), \qquad q_i \ := \ |\det(\, v_j ; \, j = 0,\dots, d, \ j \ne i\,)|.
\end{equation*}
\end{definition}

The weight systems of isomorphic lattice simplices coincide up to order. Denote by $v_0,\dots,v_d \in N$ the vertices of the lattice simplex $P \subseteq N_\RR$. The reduction $(Q_P)_{\red}$ is the unique reduced weight system satisfying
\begin{equation*}
    \sum\limits_{i=0}^d q_i v_i \ = \ 0.
\end{equation*}
Moreover, if the vertices of $P$ are primitive, then $(Q_P)_\red$ is well-formed. Following the naming convention in \cite{Ni} we call $\lambda_P := [N:N_P]$ the \emph{factor} of the lattice simplex $P \subseteq N_\RR$, where $N_P \subseteq N$ is the sublattice generated by the vertices of $P$. In \cite{Ka} it is called the \emph{multiplicity} of $P$.

\begin{lemma}{\rm{(cf. \cite[Lemma~2.4]{Co})}}
For any lattice simplex $P$ we have $\lambda_P = \lambda_{Q_P}$.
\end{lemma}

If $P$ has primitive vertices then its factor $\lambda_P$ coincides with the order of the torsion part of $\Cl(X(P))$. In particular $X(P)$ is a weighted projective space if and only if $Q_P$ is reduced. The following Theorem is a reformulation of~\cite[4.5--4.7]{Co}. Compare also~\cite[Thm.~5.4.5]{Ba94} and~\cite[Prop.~2]{BoBo}.

\begin{theorem}\label{thm:wpsfwps}
To any well-formed weight system $Q$ of length $d$ there exists a $d$-dimensional lattice simplex $P_Q \subseteq N_\RR$, unique up to an isomorphism, with $Q_{P_Q} = Q$. Any fake weighted projective space $X = X(P)$ with $(Q_P)_\red = Q$ is isomorphic to the quotient of $\PP(Q)$ by the action of the finite group $N/N_P$ corresponding to the inclusion $N_P \subseteq N$.
\end{theorem}

As an immediate consequence of Theorem \ref{thm:wpsfwps} we can relate the Gorenstein index and the anticanonical degree of a fake weighted projective space $X(P)$ to those of the weighted projective space $\PP((Q_P)_\red)$.

\begin{corollary}\label{cor:gidegwpsfwps}
Let $X = X(P)$ a $d$-dimensional fake weighted projective space and let $X' = \PP((Q_P)_\red)$ the corresponding weighted projective space. Then the Gorenstein index of $X$ is a multiple of the Gorenstein index of $X'$. Moreover we have $\lambda_P (-\mathcal{K}_X)^d = (-\mathcal{K}_{X'})^d$. In particular, $(-\mathcal{K}_X)^d = (-\mathcal{K}_{X'})^d$ holds if and only if $X$ is isomorphic to $X'$.
\end{corollary}

\begin{proof}
By Theorem~\ref{thm:wpsfwps} there is a square matrix $H$ in a lattice basis of $N$ with determinant $\lambda_P$ such that $P = H P_Q$ holds. Dualizing yields $P_Q^* = H^* P^*$. Now apply Lemma \ref{lemma:gorind} and Lemma \ref{lemma:degvol}.
\end{proof}

\section{Unit fraction partitions}\label{section:gi-uf}

To any $d$-dimensional lattice simplex $P \subseteq N_\RR$ of index $\iota$ we assign a partition of $1/\iota$ into a sum of $d+1$ unit fractions. The main result of this section is Proposition \ref{prop:uf-vol} where give a formula to compute the normalized volume of the dual polytope $P^*$ in terms of the denominators of these unit fractions.

\begin{definition}
Let $\iota \in \ZZ_{\ge 1}$. A tuple $A = (a_1,\dots,a_n) \in \ZZ^n_{\ge 1}$ is called a \emph{uf-partition of $\iota$} (\emph{of length $n$}) if the following holds:
\begin{equation*}
    \frac{1}{\iota} \ = \ \sum\limits_{k=1}^n \frac{1}{a_k}.
\end{equation*}
\end{definition}

\begin{proposition}\label{prop:P-uf}
Let $P \subseteq N_\RR$ a $d$-dimensional lattice simplex of index $\iota$ with weight system $Q_P = (q_0, \dots,q_d)$. Then
\begin{equation*}
        A(P) \ := \ \left( \frac{\iota |Q_P|}{q_0}, \dots, \frac{\iota |Q_P|}{q_d} \right)
\end{equation*}
is a uf-partition of $\iota$ of length $d+1$. We call it the uf-partition of $\iota$ associated to $P$.
\end{proposition}

\begin{proof}
We show that $A(P)$ consists of positive integers. Let $v_0,\dots,v_d \in N$ the vertices of $P$. For $0 \le i \le d$ let $F_i = \conv(v_0,\dots,\hat{v}_i,\dots,v_d)$ the $i$-th facet of $P$, where $\hat{v}_i$ means that $v_i$ is omitted. We have
\begin{equation*}
    0
    \ = \
    \sum\limits_{j=0}^d q_j \bangle{\iota u_{F_i},v_j}
    \ = \
    \bangle{\iota u_{F_i},v_i}q_i\, - \iota\! \sum\limits_{\tiny\begin{array}{l} j\! =\! 0, \\ j\! \ne\! i \end{array}}^d q_j
    \ = \
    (\bangle{\iota u_{F_i},v_i} + 1) q_i - \iota |Q_P|.
\end{equation*}
By definition of $\iota$ we have $\iota u_{F_i} \in M$. Thus $q_i$ divides $\iota |Q_P|$, so $A(P)$ consists of positive integers. Now summing over the reciprocals of $A(P)$ we see that it is in fact a uf-partition of $\iota$.
\end{proof}

\begin{proposition}\label{prop:uf-vol}
For any $d$-dimensional lattice simplex $P \subseteq N_\RR$ with associated uf-partition $A(P) = (a_0,\dots,a_d)$ of $\iota_P$ we have
\begin{equation*}
    \lambda_P \Vol(\iota_P P^*) \ = \ \frac{a_0\cdots a_d}{\lcm(a_0,\dots,a_d)}.
\end{equation*}
\end{proposition}

Proposition \ref{prop:uf-vol} generalizes \cite[Prop.~4.5.5]{Ni} to the case $\iota \ge 2$. For the proof of Proposition \ref{prop:uf-vol} and in preparation for the proof of Theorem \ref{thm:main} we extend Batyrev's correspondence between weight systems of reflexive polyhedra and uf-partitions of $1$ given in \cite[Thm.~5.4.3]{Ba94} to the case of higher indices.

\begin{definition}
The \emph{index of a weight system} $Q = (q_0,\dots,q_d)$ is the positive integer
\begin{equation*}
    \iota_Q \ := \ \min ( \, k \in \ZZ_{\ge 1} ; \, q_i \mid k |Q| \text{ for all } i = 0,\dots,d \, ).
\end{equation*}
\end{definition}

\begin{definition}
A tuple $A = (a_1,\dots,a_n) \in \ZZ^n_{\ge 1}$ is called a \emph{uf-partition} if it is a uf-partition of $\iota$ for some $\iota \in \ZZ_{\ge 1}$. For a uf-partition $A = (a_1,\dots,a_n)$ of $\iota$ we call
\begin{equation*}
    t_A \ := \ \lcm(a_1,\dots,a_n), \qquad \lambda_A \ := \ \gcd(\iota,a_1,\dots,a_n), \qquad A_{\red} \ := \ A/\lambda_A
\end{equation*}
the \emph{total weight}, the \emph{factor} and the \emph{reduction} of $A$. A uf-partition $A$ is called \emph{reduced} if it coincides with its reduction and it is called \emph{well-formed} if $a_i \mid \lcm(a_j \, ; \, j \ne i)$ holds for all $i = 1,\dots,n$
\end{definition}

\begin{proposition}\label{prop:ws-ufp}
Let $Q = (q_0,\dots,q_d)$ any weight system of length $d$ and index $\iota$ and let $A = (a_0,\dots,a_d)$ any uf-partition of length $d+1$. Then the following hold:
\begin{enumerate}
    \item $A(Q) := ( \iota |Q|/q_0, \dots, \iota |Q|/q_d )$ is a reduced uf-partition of $\iota$ of length $d+1$.
    \item $Q(A) := ( t_A/a_0, \dots, t_A/a_d )$ is a reduced weight system of length $d$.
    \item $Q(A(Q)) = Q_\red$ and $A(Q(A)) = A_\red$ and this correspondence respects well-formedness.
\end{enumerate}
\end{proposition}

For the proof of Proposition \ref{prop:ws-ufp} we need the following Lemma.

\begin{lemma}\label{lemma:detG}
For $\iota, a_1,\dots,a_n \in \ZZ$ set
\begin{equation*}
    G(\iota;a_1,\dots,a_n)
    \ := \
    \left[\begin{array}{ccccccc}
        (a_1 - \iota ) && -\iota && \dots && -\iota \\
        &&&&&&\\
        -\iota && (a_2 - \iota) && \ddots && \vdots \\
        &&&\ddots&&&\\
        \vdots && \ddots && (a_{n-1} - \iota) && -\iota \\
        &&&&&&\\
        -\iota && \dots && -\iota && (a_n - \iota)
    \end{array}\right].
\end{equation*}
Then
\begin{equation*}
    \det(G(\iota;a_1,\dots,a_n)) \ = \ a_1\cdots a_n - \iota \sum\limits_{i=1}^n \prod\limits_{j \ne i} a_j.
\end{equation*}
\end{lemma}

\begin{proof}
We prove the Lemma by induction on $n$. The cases $n=1$ and $n=2$ are verified by direct computation. Let $n \ge 3$. Subtracting the second to last row of $G := G(\iota;a_1,\dots,a_n)$ from the last row, we obtain
\begin{equation*}
    \det(G) \ = \ a_n \det(G') + a_{n-1} \det(G''),
\end{equation*}
where $G' = G(\iota; a_1,\dots,a_{n-1})$ and $G'' = G(\iota; a_1,\dots,a_{n-2},0)$. By the induction hypothesis we have
\begin{equation*}
    \det(G') \ = \ a_1\cdots a_{n-1} - \iota \sum\limits_{i=1}^{n-1} \prod\limits_{j \ne i} a_j,
    \qquad
    \det(G'') \ = \ - \iota a_1\cdots a_{n-2}.
\end{equation*}
\end{proof}

\begin{proof}[Proof of Proposition \ref{prop:ws-ufp}]
We prove (i). The weight system $Q$ is of index $\iota$, so $q_i$ divides $\iota |Q|$. Hence $A(Q)$ consists of positive integers. Summing over the reciprocals of $A(Q)$ shows that it is a uf-partition of $\iota$. Assume $A(Q)$ is not reduced and let $A' = A(Q)_\red$. Then $A'$ is a uf-partition of some $\iota' < \iota$. Thus $q_i \mid \iota'|Q|$ holds for all $i = 0,\dots,d$, contradicting the minimality of $\iota$. So $A(Q)$ is reduced. Item (ii) follows from the fact that $t_A$ is the least common multiple of $a_0,\dots,a_d$.

We prove (iii). Let $Q = (q_0,\dots,q_d)$ a weight system of length $d$ and index $\iota$ and write $ A(Q) = (a_0,\dots,a_d)$. The matrix $G = G(\iota;a_0,\dots,a_d)$ as defined in Lemma \ref{lemma:detG} has rank $d$. Both $Q$ and $Q(A(Q))$ are contained in the kernel of $G$ and the latter weight system is reduced. Thus it suffices to show that $G$ has rank $d$. Its kernel is non-trivial, so it has at most rank $d$. The inequality
\begin{equation*}
    \frac{1}{\iota} \ > \ \frac{1}{a_0} + \dots + \frac{1}{a_{d-1}}
\end{equation*}
yields $\det(G(\iota;a_0,\dots,a_{d-1})) > 0$. Hence the minor of $G$, obtained by deleting the last column and row, does not vanish, which yields $\rk(G) = d$. Now let $A = (a_0,\dots,a_d)$ a uf-partition of $\iota$ of length $d+1$. Write $Q(A) = (q_0,\dots,q_d)$ and let $A(Q) = (a_0', \dots, a_d')$. This is a uf-partition of $\iota_Q$ and we have $a_i' q_i = \iota_Q |Q|$ for all $i=0,\dots,d$. Note that $\iota_Q$ divides $\iota$. Write $\iota = \lambda \iota_Q$. We obtain
\begin{equation*}
    \lambda a_i' \ = \ \frac{\iota}{\iota_Q} \frac{\iota_Q|Q|}{q_i} \ = \ \frac{\iota \, \iota_Q t_{A(Q)} a_i}{\iota\,\iota_Q t_{A(Q)}} \ = \ a_i.
\end{equation*}
Hence $A(Q) = \lambda A'$ holds. As $A'$ is reduced, this yields $A(Q)_\red = A'$. For the last assertion in (iii) let $Q = (q_0,\dots,q_d)$ a reduced weight system of length $d$ and write $A := A(Q) = (a_0,\dots,a_d)$. We have $q_i = t_{A(Q)}/a_i$. The well-formedness of $Q$ is equivalent to saying that 
\begin{equation*}
    \prod\limits_{j \ne i} a_j \ = \ t_{A(Q)} \gcd\left( \prod\limits_{k \ne i, j} a_k; \, j \ne i \right)
\end{equation*}
holds for all $i = 0,\dots,d$. This in turn is equivalent to the well-formedness of $A(Q)$.
\end{proof}

\begin{corollary}\label{cor:ws-ufp}
For any $d$-dimensional lattice simplex $P\subseteq N_\RR$ we have $A(P)_\red = A(Q_P)$ and $\iota_P |Q_P| = \lambda_P t_{A(P)}$.
\end{corollary}

\begin{proof}
The first assertion follows from the definitions of $A(P),\, A(Q_P)$ and Proposition \ref{prop:ws-ufp} (iii). For the second assertion note that $|(Q_P)_\red| = t_{A(Q_P)}/\iota_{Q_P} = t_{A(P)}/\iota_P$ holds.
\end{proof}

\begin{proof}[Proof of Proposition \ref{prop:uf-vol}]
For any weight system $Q = (q_0,\dots, q_d)$ of length $d$ we define (cf. \cite{Co,Ni})
\begin{equation*}
    m_Q \ := \ \frac{|Q|^{d-1}}{q_0 \cdots q_d} \ \in \ \QQ_{>0}.
\end{equation*}
By \cite[Proposition~3.6]{Ni} the weight systems $Q_P$ and $Q_{\iota_P P^*}$ are related by $Q_{\iota_P P^*} = \iota_P^d m_{Q_P} Q_P$. Note that for the normalized volume of $P \subseteq N_\RR$ we have $\Vol(\iota_P P^*) = |Q_{\iota_P P^*}|$. Moreover by Corollary \ref{cor:ws-ufp} we have $\iota_P |Q_P| = \lambda_P t_{A(P)}$. We obtain
\begin{equation*}
    \lambda_P \Vol(\iota_P P^*) \ = \ \lambda_P \iota_P^d m_{Q_P} |Q_P| \ = \ \lambda_P \iota_P^d \frac{|Q_P|^d}{q_0 \cdots q_d} \ = \ \frac{a_0 \cdots a_d}{\lcm(a_0,\dots,a_d)}.
\end{equation*}
\end{proof}

\section{Sharp bounds for uf-partitions}\label{section:syl-seq}

Proposition \ref{prop:uf_ineq} gives an upper bound on the expression obtained in Proposition~\ref{prop:uf-vol}. This constitutes the number theoretic part of the proof of Theorem~\ref{thm:main}. The Lemmas thereafter are preparation for the proof of Proposition \ref{prop:uf_ineq}.

\begin{definition}\label{def:sylseq}
For any $\iota \in \ZZ_{\ge 1}$ we define a sequence $S_\iota = (s_{\iota,1}, s_{\iota,2}, \dots)$ of positive integers by
\begin{equation*}
    s_{\iota,1} \ := \ \iota+1, \qquad s_{\iota,k+1} \ := \ s_{\iota,k} (s_{\iota,k} - 1) + 1.
\end{equation*}
Moreover, for any $k \in \ZZ_{\ge 1}$ we set $t_{\iota,k} := s_{\iota,k} - 1$. We denote by $\syl_{\iota,n}$ the uf-partition of $\iota$ of length $n$ given by
\begin{equation*}
    \syl_{\iota,n} \ := \ (s_{\iota,1}, \dots, s_{\iota,n-2}, 2\, t_{\iota,n-1}, 2\, t_{\iota,n-1}).
\end{equation*}
Following the naming convention in \cite{Ni} we call $\syl_{\iota,n}$ the \emph{enlarged sylvester partition} (\emph{of $\iota$ of length $n$}).
\end{definition}

\begin{proposition}\label{prop:uf_ineq}
Let $\iota \in \ZZ_{\ge 1}$ and $n \ge 3$. Assume $(\iota,n) \ne (1,3)$. For any uf-partition $A = (a_1,\dots,a_n)$ of $\iota$ with $a_1 \le \dots \le a_n$ we have
\begin{equation*}
    \frac{a_1 \cdots a_n}{\lcm(a_1,\dots,a_n)} \ \le \ a_1 \cdots a_{n-1} \ \le \ \frac{2\, t_{\iota,n-1}^2}{\iota}.
\end{equation*}
Equality in the second case holds if and only if one of the following holds:
\begin{enumerate}
    \item[-] $(\iota,n)=(2,3)$ and $A = (6,6,6)$.
    \item[-] $(\iota,n)=(1,4)$ and $A = (2,6,6,6)$.
    \item[-] $A$ is the enlarged sylvester partition $\syl_{\iota,n}$.
\end{enumerate}
\end{proposition}

This Proposition is a generalization of \cite[Thm.~5.1.3]{Ni}. There Nill utilizes and expands the techniques of Izhboldin and Kurliandchik presented in \cite{IzKu}. Here we modify Nill's arguments to incorporate the cases for $\iota \ge 2$. Let $\iota, n \in \ZZ_{\ge 1}$. We denote by $A_\iota^n \subseteq \RR^n$ the compact set of all tuples $x \in \RR^n$ with
\begin{enumerate}
    \item[(A1)] $x_1 \ge \dots \ge x_n \ge 0$,
    \item[(A2)] $x_1 + \dots + x_n = 1/\iota$,
    \item[(A3)] $x_1\cdots x_k \le \iota (x_{k+1} + \dots + x_n)$ for all $k = 1, \dots, n-1$.
\end{enumerate}

\begin{lemma}\label{lemma:ufp_in_A}
For any uf-partition $A = (a_1,\dots,a_n)$ of $\iota$ with $a_1 \le \dots \le a_n$ the tuple $(1/a_1,\dots,1/a_n)$ is contained in $A_{\iota}^n$.
\end{lemma}

\begin{proof}
The tuple $(1/a_1,\dots,1/a_n)$ fulfills conditions (A1) and (A2). For the third condition let $1 \le k \le n-1$. Then we have
\begin{equation*}
    \iota \left( \frac{1}{a_{k+1}} + \dots + \frac{1}{a_{n}} \right) \ = \ 1 - \iota \left( \frac{1}{a_1} + \dots + \frac{1}{a_k} \right)
    \ = \ \frac{a_1 \cdots a_k - \iota\left(\sum_{j=1}^k \prod_{i \ne j} a_i\right)}{a_1 \cdots a_k}. 
\end{equation*}
The numerator on the right hand side is a positive integer. In particular, it is at least one.
\end{proof}

The main part of the proof of Proposition \ref{prop:uf_ineq} is incorporated in the following Lemma, which extends \cite[Lemma~5.6]{Ni}.

\begin{lemma}\label{lemma:izku}
Let $n \ge 3$, $\iota \in \ZZ_{\ge 1}$ and let $x \in A_\iota^n$. Then, except for the case $(\iota,n) = (1,3)$, we have
\begin{equation*}
    x_1\cdots x_{n-1} \ge \frac{\iota}{2\,t_{\iota,n-1}^2}.
\end{equation*}
Equality holds if and only if one of the following holds:
\begin{enumerate}
    \item[-] $(\iota,n) = (2,3)$ and $(x_1,x_2,x_3) = (1/6,1/6,1/6)$.
    \item[-] $(\iota,n) = (1,4)$ and $(x_1,x_2,x_3,x_4) = (1/2,1/6,1/6,1/6)$.
    \item[-] $(1/x_1,\dots,1/x_n)$ is the enlarged sylvester partition $\syl_{\iota,n}$.
\end{enumerate}
\end{lemma}

We will need the following result, which is an extension of \cite[Lemma~5.4]{Ni} to higher indices.

\begin{lemma}\label{lemma:t_n-ineq}
Let $\iota \in \ZZ_{\ge 1}$, $n \in \ZZ_{\ge 1}$ and $1 \le r \le n$. Then, except for the case $(\iota,n,r) = (1,2,2)$, we have
\begin{equation*}
    (r+1)^r\, t_{\iota,n-r+1}^{r+1} \ \le \ 2\, t_{\iota,n}^2.
\end{equation*}
Equality holds if and only if either $r = 1$ or $(\iota,n,r) = (1,3,2)$ or $(\iota,n,r) = (2,2,2)$.
\end{lemma}

\begin{proof}
We prove the Lemma by induction on $n$ and $r$. The case $r = 1$ is clear. Let $r \ge 2$. The cases $n = 2$ and $n = 3$ are verified by direct computation. Let $n \ge 4$. Then for any $2 \le r \le n$ we have $s_{\iota,n-1} > (r+1)^2/r$. Furthermore, for any $k\in \ZZ_{\ge 1}$ we have $s_{\iota,k} > (r+1)/r$. Combining these two inequalities, we obtain:
\begin{equation*}
    r \left(\frac{r+1}{r}\right)^r \ <  \ s_{\iota,n-r+1}\cdots s_{\iota,n-1}.
\end{equation*}
Moreover, $t_{\iota,n-1}^2 < t_{\iota,n}$ always holds. Now by the induction hypothesis the statement of the Lemma is true for $(\iota, n-1, r-1)$, ie. $r^{r-1}\, t_{\iota,n-r+1}^{r} \le 2\, t_{\iota,n-1}^2$ holds. Together with the previous considerations, we obtain:
\begin{equation*}
    (r+1)^r\, t_{\iota,n-r+1}^{r+1} \ \le \ 2\, t_{\iota,n-1}^2\, r\, \left( \frac{r+1}{r} \right)^r t_{\iota,n-r+1} \ < \ 2\, t_{\iota,n-1}^2\, t_{\iota,n} \ < \ 2\, t_{\iota,n}^2.
\end{equation*}
\end{proof}

\begin{lemma}\label{lemma:yn-1=yn}
Let $n \ge 3$ and let $y \in A_\iota^n$ minimizing the product $y_1\cdots y_{n-1}$. Denote by $i_0 \in \{1,\dots,n\}$ the least index with $y_{i_0} = y_n$. Then the following hold:
\begin{enumerate}
    \item $i_0 \le n-1$.
    \item For any $1 \le k \le i_0-2$ we have $y_k = 1/s_{\iota,k}$.
\end{enumerate}
\end{lemma}

\begin{proof}
We prove (i). Assume $y_{n-1} > y_n$. Choose $0 < \epsilon < (y_{n-1}-y_n)/2$. Then the tuple
\begin{equation*}
    (\tilde{y}_1,\dots,\tilde{y}_n) \ = \ (y_1,\dots,y_{n-2},y_{n-1}-\epsilon,y_n+\epsilon).
\end{equation*}
is contained in $A_\iota^n$. We have $\tilde{y}_1 \cdots \tilde{y}_{n-1} < y_1 \cdots y_{n-1}$, contradicting the minimality of $y$. Thus $y_{n-1} = y_n$ holds. We prove (ii). For this we first show that $y_k > y_{k+1}$ and $y_1\cdots y_k = \iota(y_{k+1} + \dots + y_{n})$ holds for any $1 \le k \le i_0 - 2$. Assume $y_k = y_{k+1}$. Then we can find $1 \le i \le k < j < i_0$ with $y_{i-1} > y_i = \dots = y_k = \dots = y_j > y_{j+1}$. Here $y_1\cdots y_k < \iota(y_{k+1} + \dots + y_n)$ holds, since otherwise we had $0 = y_k(\iota - y_1\cdots y_{k-1}) + \iota(y_{k+2} + \dots + y_n)$, where the right hand side is positive. We can thus find $\epsilon > 0$ such that the tuple
\begin{equation*}
    (\tilde{y}_1,\dots,\tilde{y}_n) \ = \ (y_1,\dots,y_{i-1},y_i+\epsilon,y_{i+1},\dots,y_{j-1},y_j-\epsilon,y_{j+1},\dots,y_n)
\end{equation*}
is contained in $A_\iota^n$. For the product of the first $n-1$ entries we have
\begin{equation*}
    \tilde{y}_1 \cdots \tilde{y}_{n-1} \ = \ y_1\cdots y_{n-1} \left( 1 - \frac{\epsilon^2}{y_i y_j} \right) \ < \ y_1\cdots y_{n-1},
\end{equation*}
contradicting the minimality of $y$. Hence $y_k > y_{k+1}$ holds for $k = 1,\dots,i_0-2$. Now assume that $y_1\cdots y_k < \iota(y_{k+1} + \dots + y_n)$ holds. Again, we find $\epsilon > 0$ such that the tuple
\begin{equation*}
    (\tilde{y}_1,\dots,\tilde{y}_n) \ = \ (y_1,\dots,y_{k-1},y_k+\epsilon,y_{k+1}-\epsilon,y_{k+2},\dots,y_n)
\end{equation*}
is contained in $A_\iota^n$, leading to the same contradiction as before. Hence $y_1\cdots y_k = \iota(y_{k+1} + \dots + y_n)$ holds for $k = 1,\dots,i_0-2$. Using these identities we can compute $y_k$. We have $y_1 = \iota(y_2 + \dots + y_n) = 1 - \iota y_1$. Solving this for $y_1$ we obtain $y_1 = 1/(\iota+1) = 1/s_{\iota,1}$. Proceeding in this way we obtain $y_k = 1/s_{\iota_k}$ for all $1 \le k \le i_0-2$.
\end{proof}

\begin{proof}[Proof of Lemma \ref{lemma:izku}]
Let $y \in A_\iota^n$ minimizing the product $y_1\cdots y_{n-1}$. By Lemma \ref{lemma:ufp_in_A} the tuple of reciprocals of the enlarged sylvester partition $\syl_{\iota,n}$ is contained in $A_\iota^n$. Hence
\begin{equation*}
    y_1\cdots y_{n-1} \ \le \ \frac{1}{s_{\iota,1}} \cdots \frac{1}{s_{\iota,n-2}} \cdot \frac{1}{2t_{\iota,n-1}} \ = \ \frac{\iota}{2\,t_{\iota,n-1}^2}
\end{equation*}
holds. Let $i_0 \in \{1,\dots,n\}$ the least index with $y_{i_0} = y_{n}$. By Lemma \ref{lemma:yn-1=yn} we have $i_0 \le n-1$. Set $r := n - i_0$. We distinguish three cases.

\medskip
\noindent\emph{Case 1.}
Assume $i_0 = 1$. Then $r = n-1$ and $y_k = 1/(\iota n)$ holds for all $k=1,\dots,n$. We obtain
\begin{equation*}
    \frac{\iota}{2\,t_{\iota,n-1}^2} \ \ge \ y_1 \cdots y_{n-1} \ = \ \frac{1}{(\iota n)^{n-1}} \ = \ \frac{1}{(r+1)^r t_{\iota,n-r}^r}.
\end{equation*}
Comparing this to Lemma \ref{lemma:t_n-ineq} for the case $r = n-1$, we see that this is only possible for $(\iota,n,r) = (2,3,2)$ and $(y_1, y_2, y_3) = (1/6,1/6,1/6)$ and in this case equality holds.

\medskip
\noindent\emph{Case 2.}
Assume $i_0 = 2$. Then $r = n-2$ and $y_1 > y_2 = \dots = y_n$ holds. By (A2) we obtain $y_1 = 1/\iota - (n-1)y_n$. Using this identity, together with (A3), we obtain an interval of possible values for $y_n$. On this interval we define a function $f$ by
\begin{equation*}
    f(y_n) \ := \ y_1 \cdots y_{n-1} \ = \ \left(\frac{1}{\iota} - (n-1)y_n\right) y_n^{n-2}, \qquad y_n \ \in \ \left[ \frac{1}{(r+1)t_{\iota,n-r}}, \frac{1}{\iota n} \right).
\end{equation*}
The function $f$ is monotone increasing, so it attains its minimum on the lower boundary of the interval. We obtain
\begin{equation*}
    \qquad \frac{\iota}{2\,t_{\iota,n-1}^2} \ \ge \ y_1 \cdots y_{n-1} \ = \ f(y_n) \ \ge \ \frac{\iota}{(r+1)^r t_{\iota,n-r}^{r+1}}
\end{equation*}
Comparing this to Lemma \ref{lemma:t_n-ineq} for the case $r = n-2$, this is only possible for $(\iota,n) = (1,4)$ and $(y_1,y_2,y_3,y_4) = (1/2,1/6,1/6,1/6)$, or $n = 3$ and $(1/y_1,1/y_2,1/y_3) = \syl_{\iota,3}$ and in these cases equality holds.

\medskip
\noindent\emph{Case 3.}
Assume $i_0 \ge 3$. Since $y_{n-1} = y_{n}$ holds, this case only appears for $n \ge 4$. We have $1 \le r \le n-3$. By Lemma \ref{lemma:yn-1=yn} we have $y_k = 1/s_{\iota,k}$ for all $1 \le k \le i_0-2$. Similar to the second case we use (A2) and (A3) to express $y_{i_0-1}$ in terms of $y_n$ and determine an interval of possible values for $y_n$:
\begin{equation*}
    \qquad\qquad y_{i_0-1} \ = \ \frac{1}{t_{\iota,n-r-1}} - (r+1)\,y_n, \qquad y_n \ \in \ \left[ \frac{1}{(r+1)t_{\iota,n-r}}, \frac{1}{(r+2) t_{\iota,n-r-1}} \right).
\end{equation*}
Again, we define the function $f(y_n) := y_1 \cdots y_{n-1}$ on that interval. It is monotone increasing up to some point and then it is monotone decreasing, so it attains its minimum at the boundary. We obtain:
\begin{equation*}
    \frac{\iota}{2\,t_{\iota,n-1}^2}
    \ \ge \
    y_1 \cdots y_{n-1}
    \ \ge \
    \min\left( \frac{\iota}{(r+1)^r t_{\iota,n-r}^r}, \frac{\iota}{(r+2)^{r+1} t_{\iota,n-r-1}^{r+1}} \right).
\end{equation*}
Comparing this to Lemma \ref{lemma:t_n-ineq} for $1 \le r \le n-3$, this is only possible for $r = 1$ and $y_n = 1/(2 t_{\iota,n-1})$. Hence $(1/y_1,\dots,1/y_n) = \syl_{\iota,n}$ and in this case equality holds.
\end{proof}

\begin{proof}[Proof of Proposition \ref{prop:uf_ineq}]
Let $A = (a_1,\dots,a_n)$ a uf-partition of $\iota$ with $a_1 \le \dots \le a_n$. The first inequality is due to the fact that $a_n$ divides $\lcm(a_1,\dots,a_n)$. By Lemma \ref{lemma:ufp_in_A} the tuple $x = (1/a_1,\dots,1/a_n)$ is contained in $A_\iota^n$. The second inequality and the assertions thereafter now follow immediately from Lemma \ref{lemma:izku}.
\end{proof}

\section{Proof of the main result}\label{section:proof}

We state and prove the main result of the article.

\begin{definition}
For any $d \ge 2$ and any $\iota \in \ZZ_{\ge 1}$ we denote by $Q_d^\iota$ the well-formed weight system
\begin{equation*}
    Q_{\iota,d} \ := \ Q(\syl_{\iota,d+1}) \ = \ \left( \frac{2 t_{\iota,d}}{s_{\iota,1}}, \dots , \frac{2 t_{\iota,d}}{s_{\iota,d-1}}, 1, 1\right),
\end{equation*}
where $t_{\iota,d}$ and $s_{\iota,k}$, $k = 1,\dots,d-1$ are as in Definition \ref{def:sylseq}.
\end{definition}

\begin{theorem}\label{thm:main2}
The anticanonical degree of any $d$-dimensional fake weighted projective space $X$ of Gorenstein index $\iota$ is bounded according to the following table.
\begin{longtable}{c|c|cc|cc|c}
    $d$
    &
    $1$
    &
    $2$
    &
    $2$
    &
    $3$
    &
    $3$
    &
    $\ge 4$
    \\[2pt]\hline
    &&&&&\\[-9pt]
    $\iota$
    &
    $\ge 1$
    &
    $1$
    &
    $\ge\!2$
    &
    $1$
    &
    $\ge\!2$
    &
    $\ge 1$
    \\[2pt]\hline
    &&&&&\\[-10pt]
    $\begin{array}{c}
         \text{bound on}\\
         (-\mathcal{K}_X)^d
    \end{array}$
    &
    $2$
    &
    $9$
    &
    $\frac{2(\iota+1)^2}{\iota}$
    &
    $72$
    &
    $\frac{2t_{\iota,3}^2}{\iota^4}$
    &
    $\frac{2t_{\iota,d}^2}{\iota^{d+1}}$
    \\[8pt]\hline
    &&&&&\\[-10pt]
    $\begin{array}{c}
         \text{attained}\\
         \text{exactly by}
    \end{array}$
    &
    $\PP^1$
    &
    $\PP^2$
    &
    $\PP(2\iota,1,1)$
    &
    $\begin{array}{l} \PP(3,1,1,1), \\ \PP(6,4,1,1) \end{array}$
    &
    $\PP(Q_{\iota,3})$
    &
    $\PP(Q_{\iota,d})$
\end{longtable}
\noindent Equality on the degree holds if and only if $X$ is isomorphic to one of the weighted projective spaces in the last row of the table.
\end{theorem}

\begin{proof}
Let $X$ a $d$-dimensional fake weighted projective space of Gorenstein index $\iota$. Let $P \subseteq N_\RR$ a $d$-dimensional lattice simplex with $X(P) \cong X$. Then $P$ has index $\iota$. Let $A := A(P) = (a_0,\dots,a_d)$ the uf-partition of $\iota$ associated to $P$. We may assume $a_0 \le \dots \le a_d$. By Lemma \ref{lemma:degvol} and Proposition \ref{prop:uf-vol} we have
\begin{equation*}
    (-\mathcal{K}_X)^d \ = \ \Vol(P^*) \ = \ \frac{1}{\iota^d} \Vol(\iota P^*) \ \le \ \frac{1}{\iota^d} \frac{a_0\cdots a_d}{\lcm(a_0,\dots,a_d)}.
\end{equation*}
For $d = 1$ there is only one fake weighted projective space, namely $\PP^1$, which has anticanonical degree $-\mathcal{K} = 2$. Let $d \ge 2$. In case $\iota = 1$ and $d = 2$ the right hand side of the inequality is bounded from above by $9$ and $\PP^2$ is the only Gorenstein fake weighted projective plane whose degree attains that value, see~\cite[Ex.~4.7]{Ni}. If $(\iota,d) \ne (1,2)$, then Proposition \ref{prop:uf_ineq} provides the upper bound
\begin{equation*}
    (-\mathcal{K}_X)^d \ \le \ \frac{1}{\iota^d} \frac{a_0\cdots a_d}{\lcm(a_0,\dots,a_d)} \ \le \ \frac{2 t_{\iota,d}^2}{\iota^{d+1}}.
\end{equation*}
Equality in the first case holds if and only if $X$ is a weighted projective space, see Corollary \ref{cor:gidegwpsfwps}. By Proposition \ref{prop:uf_ineq} equality in the second case holds if and only if one of the following holds:
\begin{enumerate}
    \item $(\iota,d) = (2,2)$ and $A = (6,6,6)$.
    \item $(\iota,d) = (1,3)$ and $A = (2,6,6,6)$.
    \item $A = \syl_{\iota,d+1}$.
\end{enumerate}
  Note that the uf-partition in (i) is not reduced. In particular, there is no weighted projective plane $X(P)$ of Gorenstein index $2$ with $A(P) = (6,6,6)$. The uf-partitions in (ii) and (iii) are reduced and well-formed. By Theorem \ref{thm:wpsfwps} and Proposition \ref{prop:ws-ufp} the uf-partition $A = (2,6,6,6)$ corresponds to the three-dimensional Gorenstein weighted projective space $X = \PP(3,1,1,1)$ and the uf-partition $A = \syl_{\iota,d+1}$ corresponds to the $d$-dimensional weighted projective space $X = \PP(Q_{\iota,d})$.
\end{proof}

\end{document}